\documentclass[11pt]{amsproc}

\usepackage{amsaddr}
\usepackage{amscd,amssymb}

\title{A survey of s-unital and locally unital rings}

\newtheorem{thm}{Theorem}
\newtheorem{prop}[thm]{Proposition}

\theoremstyle{definition}
\newtheorem{defi}[thm]{Definition}
\newtheorem{exa}[thm]{Example}
\newtheorem{rem}[thm]{Remark}

\begin{document}

\author{Patrik Nystedt}
\address{University West,
Department of Engineering Science, 
SE-46186 Trollh\"{a}ttan, Sweden}

\email{patrik.nystedt@hv.se} 



\begin{abstract}
We gather some classical results and examples that show strict inclusion
between the families of unital rings, rings with enough idempotents,
rings with sets of local units, locally unital rings, s-unital rings and idempotent rings.
\end{abstract}

\maketitle

In many presentations of ring theory, authors make the assumption that all rings are unital,
that is that they possess a multiplicative identity element.
There are, however, lots of natural constructions in ring theory which share all
properties of unital rings except the property of having a multipli\-cative identity.
Such constructions include ideals, infinite direct sums of rings, and linear
transformations of finite rank of an infinite dimensional vector space.
For many examples of rings lacking a multiplicative identity there still exist weaker versions of identity elements.
The purpose of the present article is to gather some classical results and
examples of rings having different degrees of weak forms of identity elements,
ordered in hierarchy. 
To be more precise, we wish to show the following strict inclusions of families of rings:
$$
\begin{array}{rcl}
\{ \mbox{unital rings} \} & \subsetneq & \{ \mbox{rings with enough idempotents} \} \\
                          & \subsetneq & \{ \mbox{rings with sets of local units} \} \\
                          & \subsetneq & \{ \mbox{locally unital rings} \} \\
                          & \subsetneq & \{ \mbox{s-unital rings} \} \\
                          & \subsetneq & \{ \mbox{idempotent rings} \} \\
                          & \subsetneq & \{ \mbox{rings} \}.
\end{array}
$$
In our presentation, we will begin with the class of rings and narrow down our
results and examples until we reach the class of unital rings.

\begin{defi}
Throughout this article, $R$ denotes an associative ring.
We do not assume that $R$ has a multiplicative identity.
Let $\mathbb{Z}$ denote the set of integers and
let $\mathbb{N}$ denote the set of positive integers.
\end{defi}

\begin{defi}
The ring $R$ is called \emph{idempotent} if $R^2 = R$.
Here $R^2$ denotes the set
of all finite sums of elements of the form $r s$
for $r ,s \in R$.
\end{defi}

\begin{exa}
It is easy to construct rings which are not idempotent.
In fact, let $A$ be any non-zero abelian group.
Define a multiplication on $A$ by saying that $ab = 0$ for all $a,b \in A$.
Then $A^2 = \{ 0 \} \neq A$.

Another generic class of examples is constructed in the following way.
If $R$ is a ring and $I$ is a two-sided ideal of $R$,
with $I^2 \subsetneq I$, then $I$ is a ring which is not idempotent.
This holds for many rings $R$, for instance when $R = \mathbb{Z}$ 
and $I$ is any non-trivial ideal of $R$.
\end{exa}

The next definition was introduced by Tominaga 
in \cite{tominaga1975} and \cite{tominaga1976}.

\begin{defi}
Let $M$ be a left (right) $R$-module.
We say that $M$ is $s$-\emph{unital} if for every $m \in M$ the relation
$m \in Rm$ ($m \in mR$) holds. 
If $M$ is an $R$-bimodule, then we say that $M$
is $s$-\emph{unital} if it is $s$-unital both as a left $R$-module
and as a right $R$-module.
The ring $R$ is said to be {\it left (right) $s$-unital}
if it is left (right) $s$-unital as a left (right) module over itself.
The ring $R$ is said to be $s$-\emph{unital} if it is $s$-unital
as a bimodule over itself.
\end{defi}

\begin{exa}
The following example shows that there exist idempotent rings that
are neither left nor right $s$-unital.
Let $G = \{ e,g \}$ denote the associative semigroup
defined by the relations $e \cdot e = e$ and $e \cdot g = g \cdot e = g \cdot g = g$.
Let $K$ denote a field and put $u = (1,0)$ and $v = (0,1)$ in $K \times K$.
Let $R$ denote the twisted semigroup ring $(K \times K)[G]$
where the multiplication is defined by
$$(x_1 + x_2 g)(y_1 + y_2 g) = x_1 y_1 + (x_1 y_2 e_2 + x_2 y_1 e_1) g$$
for $x_1,x_2,y_1,y_2 \in K \times K$.
Then $R$ is associative. Indeed, take $$x_1,x_2,y_1,y_2,z_1,z_2 \in K \times K.$$ 
A straightforward calculation shows that
$$( (x_1 + x_2 g)(y_1 + y_2 g) )(z_1 + z_2 g) = 
x_1 y_1 z_1 + (x_2 y_1 z_1 e_1 + x_1 y_1 z_2 e_2)g$$
and
$$(x_1 + x_2 g) ( (y_1 + y_2 g) (z_1 + z_2 g) ) = 
x_1 y_1 z_1 + (x_2 y_1 z_1 e_1 + x_1 y_1 z_2 e_2)g.$$
Also $R$ is neither left nor right $s$-unital.
In fact, take $x_1,y_2 \in K \times K$.
If $g (x_1 + x_2 g) = g$ then $e_1 x_1 g = g$
so that $e_1 x_1 = (1,1)$ in $K \times K$ which is a contradiction.
In the same way $(x_1 + x_2 g) g = g$ leads to
$x_1 e_2 = (1,1)$ in $K \times K$ which is a contradiction.
However, $R$ is idempotent since for all $(k,l) \in K \times K$
the following relations hold
$$(k,l)1 = (k,l) \cdot (1,1) \in R^2$$ 
and 
$$(k,l)g = (k,0)g \cdot (1,1)1 + (0,l)1 \cdot (1,1)g \in R^2.$$
\end{exa}

\begin{exa}\label{exampleleftright}
The following example (inspired by \cite[Exercise 1.10]{lam2003}) 
shows that there are lots of examples of
rings which are left (right) $s$-unital but not right (left) $s$-unital.
Let $A$ be a unital ring with a non-zero multiplicative identity $1$.

(a) Let $B_l$ denote the set $A \times A$
equipped with componentwise addition and multiplication defined
by the relations 
$$(a , b)(c , d) = (ac , ad)$$
for $a,b,c,d \in A$. 
Now we show that $B_l$ is associative.
Take $a,b,c,d,e,f \in A$. Then
$$( (a,b)(c,d) ) (e,f) = (ac,ad)(e,f) = (ace , acf)$$
and 
$$(a,b) ( (c,d)(e,f) ) = (a,b)(ce,cf) = (ace , acf).$$
It is clear that any element of the form
$(1,a)$, for $a \in A$, is a left identity for $B_l$.
However, $B_l$ is not right unital.
Indeed, since $(0,1) \notin \{ (0,0) \} = (0,1) B_l$
it follows that $B_l$ is not even right $s$-unital.
For each $n \in \mathbb{N}$ let $C_n$ denote a copy of $B_l$
and put $C = \oplus_{n \in \mathbb{N}} C_n$. Then $C$ is left $s$-unital
but not left unital. Since none of the $C_n$ are right $s$-unital
it follows that $C$ is not right $s$-unital.

(b) Let $B_r$ denote the set $A \times A$
equipped with componentwise addition and multiplication defined
by the relations 
$$(a , b)(c , d) = (ac , b c)$$
for $a,b,c,d \in A$. 
Now we show that $B_r$ is associative.
Take $a,b,c,d,e,f \in A$. Then
$$( (a,b)(c,d) ) (e,f) = (ac,bc)(e,f) = (ace , bce)$$
and
$$(a,b) ( (c,d)(e,f) ) = (a,b) (ce , de) = (ace , bce).$$ 
It is clear that any element of the form
$(1,a)$, for $a \in A$, is a right identity for $B_l$.
However, $B_r$ is not left unital.
Indeed, since $(0,1) \notin \{ (0,0) \} = B_r (0,1)$
it follows that $B_r$ is not even left $s$-unital.
For each $n \in \mathbb{N}$ let $D_n$ denote a copy of $B_r$
and put $D = \oplus_{n \in \mathbb{N}} D_n$. Then $D$ is right $s$-unital
but not right unital. Since none of the $D_n$ are left $s$-unital
it follows that $D$ is not left $s$-unital.
\end{exa}

\begin{defi}\label{defvee}
If $e' , e'' \in R$, then put $e' \vee e'' = e' + e'' - e' e''$.
\end{defi}

\begin{prop}\label{proptominaga}
Let $M$ be a left (right) $R$-module.
Then $M$ is left (right) $s$-unital if and only if 
for all $n \in \mathbb{N}$ and
all $m_1,\ldots,m_n \in M$ there is $e \in R$ such that
for all $i \in \{ 1,\ldots,n \}$ the relation $e m_i = m_i$ ($m_i e = m_i$) holds.
\end{prop}

\begin{proof}
We follow the proof of \cite[Theorem 1]{tominaga1976}.
The ``if'' statements are trivial.
Now we show the ``only if'' statements.

First suppose that $M$ is a left $R$-module which is $s$-unital.
Take $n \in \mathbb{N}$ and $m_1,\ldots,m_n \in M$.
Take $e_n \in R$ such that $e_n m_n = m_n$
and for every $i \in \{ 1, \ldots , n-1 \}$ put $v_i = m_i - e_n m_i$.
By induction there is an element $e' \in R$ such that
for every $i \in \{ 1, \ldots , n-1 \}$ the equality $e' v_i = v_i$ holds.
Put $e = e' \vee e_n$. Then
$$e m_n = 
e' m_n + e_n m_n - e' e_n m_n =
e' m_n + m_n - e' m_n = m_n$$
and for every $i \in \{ 1, \ldots , n-1 \}$ we get that
$$
\begin{array}{rcl}
e m_i & = & e' m_i + e_n m_i - e' e_n m_i \\ 
      & = & e' (m_i - e_n m_i) + e_n m_i \\
      & = & e' v_i + e_n m_i \\
      & = & v_i + e_n m_i \\
      & = & m_i - e_n m_i + e_n m_i \\
      & = & m_i.
\end{array}
$$
Now suppose that $M$ is a right $R$-module which is $s$-unital.
Take $n \in \mathbb{N}$ and $m_1,\ldots,m_n \in M$.
Take $e_n \in R$ such that $m_n e_n = m_n$
and for every $i \in \{ 1, \ldots , n-1 \}$ put $v_i = m_i - m_i e_n$.
By induction there is an element $e' \in R$ such that
for every $i \in \{ 1, \ldots , n-1 \}$ the equality $v_i e' = v_i$ holds.
Put $e = e_n \vee e'$. Then
$$m_n e = 
m_n e' + m_n e_n - m_n e_n e' =
m_n e' + m_n - m_n e' = m_n$$ 
and for every $i \in \{ 1, \ldots , n-1 \}$ we get that
$$
\begin{array}{rcl}
m_i e & = & m_i e' + m_i e_n - m_i e_n e' \\
      & = & (m_i - m_i e_n) e' + m_i e_n \\
      & = & v_i e' + m_i e_n \\
      & = & v_i + m_i e_n \\
      & = & m_i - m_i e_n + m_i e_n \\
      & = & m_i.
\end{array}
$$     
\end{proof}

\begin{prop}\label{e'e''ebimodule}
Let $M$ be an $R$-bimodule and suppose that $e',e'' \in R$.
Let $X$ be a subset of $M$ such that for all $m \in X$ the relations 
$e' m = m e'' = m$ hold, then for all $m \in X$ the following relations hold 
$$(e'' \vee e') m = m (e'' \vee e') = m.$$ 
\end{prop}

\begin{proof}
This is essentially the proof of \cite[Lemma 1]{mogami1978}.
Take $m \in X$. Then
$$(e'' \vee e') m = 
(e' + e'' - e'' e')m =
e' m + e'' m - e'' e' m =
m + e'' m - e'' m = m$$ 
and
$$m (e'' \vee e') = 
m(e' + e'' - e'' e') =
m e' + m e'' - m e'' e' =
m e' + m - m e' = m.$$
\end{proof}

\begin{prop}\label{propmogami}
Let $M$ be an $R$-bimodule.
Then $M$ is $s$-unital if and only if 
for all $n \in \mathbb{N}$ and
all $m_1,\ldots,m_n \in M$ there is $e \in R$ such that
for all $i \in \{ 1,\ldots,n \}$ the relation $e m_i = m_i e = m_i$ holds.
\end{prop}

\begin{proof}
The ``if'' statement is trivial.
Now we show the ``only if'' statement.
Take $n \in \mathbb{N}$ and $m_1,\ldots,m_n \in M$.
From Proposition \ref{proptominaga} it follows that
there are $e' , e'' \in R$ such that for all $i \in \{ 1, \ldots, n \}$
the relations $e' m_i = m_i e'' = m_i$ hold.
The claim now follows from Proposition \ref{e'e''ebimodule}
if we put $e = e'' \vee e'$ and $X = \{ m_1 , \ldots , m_n \}$.
\end{proof}

\begin{prop}
The ring $R$ is left (right) $s$-unital if and only if
for all $n \in \mathbb{N}$ and all $r_1 , \ldots , r_n \in R$ there is $e \in R$ such that
for all $i \in \{ 1,\ldots,n \}$ the relation $e r_i = r_i$ ($r_i e = r_i$) holds.
\end{prop}

\begin{proof}
This follows from Proposition \ref{proptominaga}.
\end{proof}

\begin{prop}
The ring $R$ is $s$-unital if and only if
for all $n \in \mathbb{N}$ and all $r_1 , \ldots , r_n \in R$ there is $e \in R$ such that
for all $i \in \{ 1,\ldots,n \}$ the relations $e r_i = r_i e = r_i$ hold.
\end{prop}

\begin{proof}
This follows from Proposition \ref{propmogami}. 
\end{proof}

\begin{defi}
An element $e \in R$ is called \emph{idempotent} if $e^2 = e$.
\end{defi}

\begin{defi}\label{deflocallyunital}
We say that $R$ is {\it left (right) locally unital} if for all $n \in \mathbf{N}$
and all $r_1,\ldots,r_n \in R$ there is an idempotent $e \in R$
such that for all $i \in \{ 1,\ldots,n \}$ the equality
$e r_i = r_i$ ($r_i e = r_i$) holds.
We say that $R$ is {\it locally unital} if it is both left locally unital
and right locally unital.
\end{defi}

\begin{exa}
Let $R$ denote the ring of real valued continuous functions
on the real line with compact support.
Then $R$ is $s$-unital but neither left nor right locally unital.
\end{exa}

The next definition was introduced by \'{A}nh and M\'{a}rki in \cite{anh1987}.

\begin{defi}\label{defanhmarki}
The ring $R$ is said to be {\it locally unital} if
for all $n \in \mathbb{N}$ and all $r_1 , \ldots , r_n \in R$
there is an idempotent $e \in R$ such that
for all $i \in \{ 1,\ldots,n \}$ the equalities
$e r_i = r_i e = r_i$ hold.
\end{defi}

\begin{prop}\label{idempotent}
Suppose that $e',e'' \in R$ are idempotents and put $e = e'' \vee e'$.
Then 
$e^2 = e + e' e'' - e' e'' e' - e'' e' e'' + e'' e' e'' e'.$
If either of the following equalities hold 
\begin{itemize}

\item[(i)] $e' e'' = e'$ 

\item[(ii)] $e' e'' = e''$

\item[(iii)] $e'' e' = e''$

\item[(iv)] $e'' e' = e'$

\item[(v)] $e' e'' = e'' e'$ 

\end{itemize}
then $e$ is idempotent.
\end{prop}

\begin{proof}
A straightforward calculation shows that
$$
\begin{array}{rcl}
e^2 & = & ( e' + e'' - e'' e' )^2 \\
    & = & (e')^2 + e' e'' - e' e'' e' + e'' e' + (e'')^2 - (e'')^2 e' - e'' (e')^2 - e'' e' e'' + e'' e' e'' e' \\
    & = & e' + e' e'' - e' e'' e' + e'' e' + e'' - e'' e' - e'' e' - e'' e' e'' + e'' e' e'' e' \\
    & = & e + e' e'' - e' e'' e' - e'' e' e'' + e'' e' e'' e'.
\end{array}
$$   
Now we show the last part.
If (i) holds, then 
$$e' e'' - e' e'' e' - e'' e' e'' + e'' e' e'' e' =
e' - (e')^2 - e'' e' + e'' (e')^2 =
e' - e' - e'' e' + e'' e' = 0.$$
If (ii) holds, then
$$e' e'' - e' e'' e' - e'' e' e'' + e'' e' e'' e' =
e'' - (e'')^2 - (e'')^2 + (e'')^3 =
e'' - e'' - e'' + e'' = 0.$$
If (iii) holds, then 
$$e' e'' - e' e'' e' - e'' e' e'' + e'' e' e'' e' =
e' e'' - e' e'' - (e'')^2 + (e'')^2 =
-e'' + e'' = 0.$$
If (iv) holds, then
$$e' e'' - e' e'' e' - e'' e' e'' + e'' e' e'' e' =
e' e'' - (e')^2 - e' e'' + (e')^2 =
-e' + e' = 0.$$
If (v) holds, then
$$
\begin{array}{rcl}
e' e'' - e' e'' e' - e'' e' e'' + e'' e' e'' e' & = & e' e'' - (e')^2 e'' - e' (e'')^2 + (e')^2 (e'')^2 \\
                                                & = & e' e'' - e' e'' - e' e'' + e' e'' \\
                                                & = & 0.
\end{array}
$$
\end{proof}

\begin{prop}
A ring is locally unital in the sense of Definition \ref{deflocallyunital} 
if and only if it is locally unital in the sense Definition \ref{defanhmarki}.
\end{prop}

\begin{proof}
The ``only if'' statement is immediate.
Now we show the ``if'' statement.
We use the argument from the proof of \cite[Proposition 1.10]{komatsu1986}
(see also \cite[Example 1]{anh1987}).
Suppose that $R$ is a ring which is locally unital 
in the sense of Definition \ref{deflocallyunital}.
Take $n \in \mathbb{N}$ and $r_1,\ldots,r_n \in R$.
Since $R$ is right locally unital, there is an idempotent $e' \in R$
such that for all $i \in \{ 1,\ldots, n \}$ the
equality $r_i e' = r_i$ holds.
Since $R$ is left locally unital, there is an idempotent $e'' \in R$
such that $e'' e' = e'$ and for all $i \in \{ 1,\ldots,n \}$ the
equality $e'' r_i = r_i$ holds.
Put $e = e' \vee e''$. From Proposition \ref{idempotent}
it follows that $e$ is idempotent.
From Proposition \ref{e'e''ebimodule}, with $X = \{ r_1,\ldots,r_n \}$,
it follows that for all $i \in \{ 1,\ldots,n \}$ the equalities
$e r_i = r_i e = r_i$ hold.
So $R$ is locally unital in the sense of Definition \ref{defanhmarki}.
\end{proof}

\begin{defi}
The ring $R$ is called \emph{regular} if for every $r \in R$ there is
$s \in R$ such that $r = rsr$. 
\end{defi}

The next proposition is \cite[Example 1]{anh1987}.

\begin{prop}
Every regular ring is locally unital.
\end{prop}

\begin{proof}
We proceed in almost the same way as in the proof of Proposition \ref{proptominaga}.
Let $R$ be a regular ring.
Take $n \in \mathbb{N}$ and $r_1 , \ldots r_n \in R$.
First we show that $R$ is left locally unital. 
By induction there is an idempotent $e_1 \in R$ such that
for all $i \in \{ 1, \ldots , n-1 \}$ the equality $e_1 r_i = r_i$ holds.
Put $s = r_n - e_1 r_n$. Since $R$ is regular, there is $t \in R$
such that $s = sts$. Put $f = st$. Then $f$ is idempotent and
$$e_1 f = 
e_1 st = 
e_1 (r_n - e_1 r_n) t = 
(e_1 r_n - e_1^2 r_n) t =
(e_1 r_n - e_1 r_n) t = 0.$$
Put $g = f - f e_1$. 
Then $e_1 g = g e_1 = 0$ and
$$g^2 = f^2 - f^2 e_1 - f e_1 f + f e_1 f e_1 = f - f e_1 = g.$$
Let $e = e_1 + g$. Then $e$ is an idempotent.
Take $i \in \{ 1, \ldots , n-1 \}$. Then 
$$e r_i = 
(e_1 + g) r_i = 
(e_1 + g) e_1 r_i =
(e_1^2 + g e_1) r_i =
e_1 r_i = r_i.$$
Finally 
$$
\begin{array}{rcl}
e r_n & = & (e_1 + g) r_n \\
      & = & e_1 r_n + g r_n \\ 
      & = & e_1 r_n + (f - f e_1) r_n \\
      & = & e_1 r_n + f r_n - f e_1 r_n \\
      & = & e_1 r_n + f s \\
      & = & e_1 r_n + s t s \\
      & = & e_1 r_n + s \\
      & = & e_1 r_n + r_n - e_1 r_n \\
      & = & r_n.
\end{array}
$$
Now we show that $R$ is right locally unital.
By induction there is an idempotent $e_1 \in R$ such that
for all $i \in \{ 1, \ldots , n-1 \}$ the equality $r_i e_1 = r_i$ holds.
Put $s = r_n - r_n e_1$. Since $R$ is regular, there is $t \in R$
such that $s = sts$. Put $f = ts$. Then $f$ is idempotent and
$$f e_1 = 
ts e_1 = 
t (r_n - r_n e_1) e_1 = 
t (r_n e_1 - r_n e_1^2 ) =
t (r_n e_1 - r_n e_1) = 0.$$
Put $g = f - e_1 f$. 
Then $e_1 g = g e_1 = 0$ and
$$g^2 = f^2 - e_1 f^2 - f e_1 f + e_1 f e_1 f = f - e_1 f = g.$$
Let $e = e_1 + g$. Then $e$ is an idempotent.
Take $i \in \{ 1, \ldots , n-1 \}$. Then 
$$r_i e = 
r_i (e_1 + g) = 
r_i e_1 (e_1 + g) =
r_i (e_1^2 + e_1 g ) =
r_i e_1 = r_i.$$
Finally 
$$
\begin{array}{rcl}
r_n e & = & r_n (e_1 + g) \\
      & = & r_n e_1 + r_n g \\
      & = & r_n e_1 + r_n (f - e_1 f ) \\
      & = & r_n e_1 + r_n f - r_n e_1 f \\
      & = & r_n e_1 + s f \\
      & = & r_n e_1 + s t s \\
      & = & r_n e_1 + s \\
      & = & r_n e_1 + r_n - r_n e_1 \\
      & = & r_n.
\end{array}
$$
\end{proof}

The next definition was introduced by Abrams in \cite{abrams1983}.

\begin{defi}\label{defabrams}
Suppose that $E$ is a set of commuting idempotents in $R$
which is closed under the operation $\vee$ from Definition \ref{defvee}.
Then $E$ is called a {\it set of local units} for $R$ if for all
$r \in R$ there is $e \in E$ such that $er = re = r$.
\end{defi}

\begin{rem}
In \cite[Definition 1.1]{abrams1983} the condition that $E$ is closed
under $\vee$ was not included. However, since this was intended 
(personal communication with G. Abrams) we chose to include it here. 
\end{rem}

\begin{prop}\label{propE}
If $R$ has a set of local units $E$, then for all $n \in \mathbf{N}$
and all $r_1 , \ldots , r_n \in R$ there is $e \in E$ such that
for all $i \in \{ 1, \ldots , n \}$ the equalities $e r_i = r_i e = r_i$ holds.
\end{prop}

\begin{proof}
Take $n \in \mathbb{N}$ and $r_1 , \ldots r_n \in R$. 
By induction there is $e_1 , e_2 \in E$ such 
$e_2 r_n = r_n e_2 = r_n$ and for all 
$i \in \{ 1, \ldots , n-1 \}$ the relations
$e_1 r_i = r_i e_1 = r_i$ hold. Put $e = e_1 \vee e_2$. 
Then, since $e_1 e_2 = e_2 e_1$, we get that
$$e r_n = 
e_1 r_n + e_2 r_n - e_1 e_2 r_n =
e_1 r_n + r_n - e_1 r_n = 
r_n$$ 
and
$$r_n e =
r_n e_1 + r_n e_2 - r_n e_2 e_1 =
r_n e_1 + r_n - r_n e_1 = 
r_n$$
and for all $i \in \{ 1, \ldots , n-1 \}$ we get that
$$e r_i = 
e_1 r_i + e_2 r_i - e_2 e_1 r_i =
r_i + e_2 r_i - e_2 r_i = 
r_i$$
and
$$r_i e =
r_i e_1 + r_i e_2 - r_i e_1 e_2 =
r_i + r_i e_2 - r_i e_2 =
r_i.$$  
\end{proof}

\begin{prop}
If a ring has a set of local units, then it is locally unital.
\end{prop}

\begin{proof}
This follows from Proposition \ref{propE}.
\end{proof}

\begin{exa}
According to \cite[Example 1]{anh1987} 
there are regular rings that do not possess sets of local units 
in the sense of Definition \ref{defabrams}. 
\end{exa}

\begin{defi}
If $e,f \in R$ are idempotent, then $e$ and $f$ are said to be
\emph{orthogonal} if $ef = fe = 0$.
\end{defi}

The following definition was introduced by Fuller in \cite{fuller1976}.

\begin{defi}\label{deffuller}
The ring $R$ is said to have {\it enough idempotents} in case 
there exists a set $\{ e_i \}_{i \in I}$ of orthogonal
idempotents in $R$ (called a complete set of idempotents for $R$) 
such that $R = \oplus_{i \in I} R e_i = \oplus_{i \in I} e_i R$.
\end{defi}

\begin{exa}
There exist rings which have sets of local units in the sense
of Definition \ref{defabrams} but which does not have enough idempotents
in the sense of Definition \ref{deffuller}.
To exemplify this we recall the construction from \cite[Example 1.6]{abrams1981}.
Let $F$ denote the field with two elements and let
$R$ be the ring of all functions $f : \mathbb{N} \rightarrow F$.
For each $n \in \mathbb{N}$ define $f_n \in R$ by
$f_n(n) = 1$, and $f_n(m) = 0$, if $m \neq n$.
For all finite subsets $S$ of $\mathbb{N}$ define
$f_S \in R$ via $f_S = \sum_{n \in S} f_n$.
Then $I = \{ f_S \mid \mbox{$S$ is a finite subset of $\mathbb{N}$} \}$
is an ideal of $R$.
Since $R$ is unital, Zorn's lemma implies the existence
of a maximal proper ideal $M$ of $R$ with $I \subseteq M$.
Since all elements in $R$, and hence also in $M$, are idempotent,
it follows that $M$ is a ring with $E=M$ as a set of local units. 
Seeking a contradiction, suppose that $M$ has a complete set 
of idempotents $\{ e_j \}_{j \in J}$.
Since $I$, and hence $M$, contains all $f_n$, for $n \in \mathbb{N}$,
it follows that $1_R = \sum_{j \in J} e_j$.
Since $M$ is a proper ideal, we get that $1_R \notin M$
and thus it follows that $J$ is an infinite set.
Choose any partition $J = K \cup L$, with $K \cap L = \emptyset$,
and $K$ and $L$ infinite.
Define $e_K = \sum_{k \in K} e_k$ and $e_L = \sum_{l \in L} e_l$.
Since the $e_j$ are pairwise orthogonal, we get that $e_K e_L = 0$.
But $M$ is a maximal ideal of $R$. Therefore $M$ is a prime ideal of $R$
and thus $e_K \in M$ or $e_L \in M$.
Suppose that $e_K \in M$. Since $\{ e_j \}_{j \in J}$ is a complete set
of idempotents, there must exist a finite set $J'$ of $J$
with $e_K = \sum_{j \in J'} e_j$ which is a contradiction.
Analogously, the case when $e_L \in M$ leads to a contradiction.
Therefore, $M$ is not a ring with enough idempotents.
\end{exa}

\begin{defi}
If $M$ is a left (right) $R$-module,
then $M$ is called {\it left (right) unital} if there is
$e \in R$ such that for all $m \in M$ the relation 
$e m = m$ ($m e = m$) holds. 
In that case $e$ is said to be a {\it left (right) identity} for $M$.
If $M$ is an $R$-bimodule, then $M$ is called {\it unital}
if it is unital both as a left $R$-module and a right $R$-module.
The ring $R$ is said to be {\it left (right) unital} if it is
left (right) unital as a left (right) module over itself.
The ring $R$ is called {\it unital} if it is unital
as a bimodule over itself.
\end{defi}

\begin{exa}
The ring $B_l$ (or $B_r$) from Example \ref{exampleleftright}
is a ring which is left (or right) unital but not right (or left) unital.
\end{exa}

\begin{exa}
There are many classes of rings that are neither left nor right unital but still 
have enough idempotents. Here are some examples:
\begin{itemize}

\item infinite direct sums of unital rings;

\item category rings where the category has infinitely many objects (see e.g. \cite[Proposition 4]{lundstrom2006});

\item Leavitt path algebras with infinitely many vertices (see e.g. \cite[Lemma 1.2.12(iv)]{abrams2017}).

\end{itemize}
\end{exa}

\begin{prop}\label{emmem}
Let $M$ be an $R$-bimodule. Then $M$ is unital if and only if
there is $e \in R$ such that for all $m \in M$ the relations
$em = me = m$ hold.
\end{prop}

\begin{proof}
The ``if'' statement is trivial.
The ``only if'' statement follows
from Proposition \ref{e'e''ebimodule} if we put $X = M$.
\end{proof}

\begin{prop}\label{errer}
The ring $R$ is unital if and only if there is $e \in R$
such that for all $r \in R$ the relations $er = re = r$ hold.
\end{prop}

\begin{proof}
This follows from Proposition \ref{emmem} if we put $M = R$. 
\end{proof}

\begin{rem}\label{remarkidentity}
Proposition \ref{errer} can of course be proved directly
in the following way. Let $e'$ (or $e''$) be a left (or right)
identity for $R$ as a left (or right) module over itself.
Then $e' = e' e'' = e''$. 
\end{rem}

We end the article with the following remark which connects
unitality and s-unitality.

\begin{prop}
If $R$ is left (right) $s$-unital and right (left) unital,
then $R$ is unital.
\end{prop}

\begin{proof}
First suppose that $R$ is left $s$-unital and right unital.
Let $f$ be a right identity of $R$ and take $r \in R$.
From Proposition \ref{proptominaga} it follows that 
there is $e \in R$ with $er = r$ and $ef = f$.
But since $f$ is a right identity of $R$ it follows that $ef = e$.
Thus $e = f$ and hence $fr = er = r$ so that $f$ is a left identity of $R$.
Now suppose that $R$ is right $s$-unital and left unital.
Let $f$ be a left identity of $R$ and take $r \in R$.
From Proposition \ref{proptominaga} it follows that 
there is $e \in R$ with $re = r$ and $fe = f$.
But since $f$ is a left identity of $R$ it follows that $fe = e$.
Thus $e = f$ and hence $rf = re = r$ so that $f$ is a right identity of $R$.
\end{proof}

\end{document}